\documentclass[a4paper,12pt]{amsart}

\usepackage{amsmath,amssymb}
\usepackage{amsthm}
\usepackage[alphabetic]{amsrefs}

\usepackage[english]{babel}

\newtheorem{theorem}{Theorem}
\newtheorem{lemma}[theorem]{Lemma}
\newtheorem{proposition}[theorem]{Proposition}
\newtheorem{corollary}[theorem]{Corollary}
\theoremstyle{definition}
\newtheorem{definition}[theorem]{Definition}
\newtheorem{remark}[theorem]{Remark}

\newcommand{\NN}{\mathbb{N}}
\newcommand{\ZZ}{\mathbb{Z}}
\newcommand{\RR}{\mathbb{R}}
\newcommand{\CC}{\mathbb{C}}
\newcommand{\cont}{\mathcal{C}}
\newcommand{\bdry}{\partial}

\usepackage{etoolbox}
\apptocmd{\sloppy}{\hbadness 10000\relax}{}{}

\title[Connectedness of the boundary of $q$-complete domains]{On the connectedness of the boundary of $q$-complete domains}
\author{Rafael B. Andrist}
\address{Faculty of Mathematics and Physics, 
University of Ljubljana, Ljubljana, Slovenia}
\email{rafael-benedikt.andrist@fmf.uni-lj.si}

\begin{document}

\begin{abstract}
The boundary of every relatively compact Stein domain in a complex manifold of dimension at least two is connected. No assumptions on the boundary regularity are necessary. The same proofs hold also for $q$-complete domains, and in the context of almost complex manifolds as well.
\end{abstract}



\maketitle

\section{Introduction}
It seems to have been well-known to the experts in the 1980s that every bounded strictly pseudoconvex domain with $\cont^2$-smooth boundary in $\CC^n, n \geq 2,$ has connected boundary. In fact, already in 1953 Serre pointed out that every Stein manifold of dimension at least $2$ has only one end (see Section \ref{sec:ends}). For a relatively compact Stein domain that admits a collar, this would already imply connectedness of the boundary. 
However, to the best know\-ledge of the author, the earliest publication mentioning that every bounded strictly pseudoconvex domain with $\cont^2$-smooth boundary in $\CC^n, n \geq 2,$ has connected boundary, is due Rosay and Stout \cite{MR0964461}*{Corollary, p.~1018} in 1989 where they actually prove a stronger result. Again, the connectedness of the boundary was noted by Balogh and Bonk \cite{MR1793800}*{p.~513} under the same assumptions. 
In the monograph of Stout \cite{MR2305474}*{Corollary 2.4.7} it was established that every relatively compact, strictly pseudoconvex domain with $\cont^2$-smooth boundary in a Stein manifold of dimension at least two has connected boundary. The proof given there follows from a theorem of Forstneri\v{c} \cite{MR1278436} about complements of Runge domains.

\smallskip

The almost complex case, but with a $\cont^\infty$-smooth defining $J$-pluri\-sub\-har\-monic function on a neighborhood of the closure of the relatively compact domain is treated by Bertrand and Gaussier \cite{MR3359581}.

\smallskip

For the special case of bounded pseudoconvex domains in $\CC^n$, a proof without any assumptions on the boundary regularity can be found in the second edition of the textbook of Jarnicki and Pflug \cite{MR4201928}*{Corollary 2.6.10}. It relies mainly on a topological argument that was provided by Czarnecki, Kulczycki and Lubawski \cite{MR2855300} with an elementary proof: For a bounded domain in $\RR^n$ the connectedness of its complement is equivalent to the connectedness of its boundary. A very similar proof is given by Izzo \cite{Izzo} who uses an elegant homology argument that relies only on $H_1(\RR^n, \ZZ) = 0$ to obtain the above-mentioned topological fact.

\smallskip

We will prove the natural generalization of these results without any assumption on boundary regularity and on the ambient space.

\begin{definition}[\cite{MR3700709}*{Definition 3.1.3}]
Let $X$ be a complex mani\-fold. We say that a domain $\Omega \subset X$ is \emph{$q$-complete}, if there exists a $\cont^2$-smooth exhaustion function $\varphi \colon \Omega \to [0, +\infty)$ that is $q$-convex, i.e.\ if its Levi form has at most $q-1$ negative or zero eigenvalues at each point of $\Omega$.
\end{definition}

\begin{theorem}
\label{thm1}
Let $X$ be a complex manifold with $\dim_\CC X = n$ and let $\Omega \subset X$ be a relatively compact $q$-complete domain with $n > q$. Then the boundary $\bdry \Omega$ is connected.
\end{theorem}

By a result of Grauert \cite{MR0098847} $\Omega$ is Stein if and only if it is $1$-complete, and thus we obtain the following corollary.

\begin{corollary}
Let $X$ be a complex manifold with $\dim_\CC X = n > 1$ and let $\Omega \subset X$ be a relatively compact Stein domain. Then the boundary $\bdry \Omega$ is connected.
\end{corollary}

Since a domain $\Omega \subset X$ in a Stein manifold $X$ is a domain of holomorphy if and only if it is pseudoconvex, we also arrive at the next corollary.

\begin{corollary}
Let $X$ be a Stein manifold with $\dim_\CC X = n > 1$ and let $\Omega \subset X$ be a relatively compact domain of holomorphy. Then the boundary $\bdry \Omega$ is connected.
\end{corollary}

The proof also extends to the almost complex situation.

\begin{theorem}
\label{thm2}
Let $(X,J)$ be an almost complex manifold with $\dim_\RR X \geq 4$, and let $\Omega \subset X$ be a relatively compact domain with a $\cont^2$-smooth strictly $J$-plurisubharmonic exhaustion function. Then the boundary $\bdry \Omega$ is connected.
\end{theorem}

\begin{remark} The result is sharp in the following sense:
\begin{enumerate}
\item In one dimension, every bounded domain of $\CC$ is pseudoconvex --- and in fact a domain of holomorphy, but obviously does not need to have connected boundary. If we assume in addition that the domain is simply connected, its boundary will be connected, but in general not path-connected. 
\item Unbounded domains, even when simply connected, do not need to have a connected complement, e.g.\ an infinite strip in $\CC$. By taking direct products, this yields also counterexamples in higher dimensions.
\end{enumerate}
\end{remark}

\begin{remark}
For symplectic manifolds with contact type boundaries, the boundary does not need to be connected, see McDuff \cite{MR1091622} and Geiges \cite{MR1328705}.
\end{remark}

This short note is organized as follows:
For completeness, in Section \ref{sec:classical} we first give the proof for relatively compact, strictly pseudoconvex domains with $\cont^2$-smooth boundary in a Stein manifold. We could not find a reference for this proof which was communicated by Franc Forstneri\v{c}.

In Section \ref{sec:ends} we provide the topological background for the theory of ends and of continua. In Section \ref{sec:psc} we give the proofs for the general situation.

\section{A Classical Proof for Smooth Boundary}

\label{sec:classical}

The following proof for the situation where $\Omega$ is a relatively compact, strongly pseudoconvex domain with $\cont^2$-smooth boundary in a Stein manifold was communicated by Franc Forstneri\v{c}. This is likely the classical proof that was known at least since the 1980s:

\smallskip

Consider a relatively compact domain $\Omega= \{ \rho < c \}$ in an $n$-di\-men\-sion\-al Stein manifold $X$, where $n > 1$ and $\rho$ is a strongly plurisubharmonic Morse function on a neighborhood of $\overline{\Omega}$, with $d\rho \ne 0$ on $\bdry \Omega$. Let $\Omega_t = \{\rho \le t \}$ for $t \le c$. The topology of $\Omega_t$ and $\bdry \Omega_t$ only changes when $t$ passes a critical level set of $\rho$. (The local normal form of a strongly plurisubharmonic Morse function at an isolated critical point is known in principle since 1924 due to a result of Takagi \cite{zbMATH02595788} which was recovered again by Schur \cite{zbMATH03097391} and by Harvey and Wells \cite{MR0330510}.) When passing a local minimum, a new connected component of $\Omega_t$ appears, which is not a concern. The only other type of points which can change connectivity of $\Omega_t$ is a critical point of index $1$. At such a point, we add a $1$-handle to $\Omega_t$. There are two possibilities -- either this handle is attached with both ends to the same component of $\Omega_t$, or it joins two distinct components. In both cases we see by inspection that the boundary of any connected component remains connected.

Handles of higher index up to $2n-2$ do not change connectivity of the domain or its boundary: 
The core of the handle of index $k$ is a $k$-disc which is attached with its boundary $(k-1)$-sphere to $\bdry \Omega_t$. Removing a submanifold of real codimension $\ge 2$ from a manifold does not disconnect the manifold: In our case, we apply this argument to the boundary sphere of the $k$-disc as a submanifold of $\bdry \Omega_t$. 

The connectivity of $\bdry \Omega_t$ would change by attaching a handle of index $2n-1$ (i.e.\ of real codimension $1$), but such is not allowed since $\rho$ is plurisubharmonic and $n \geq 2$.

\section{Ends of topological spaces}
\label{sec:ends}

The definition of \emph{ends} of a topological space goes back to Freudenthal \cite{MR1545233} and has led to a well developed theory, see the textbook of Hughes and Ranicki \cite{MR1410261}.

\begin{definition} \hfill
\begin{enumerate}
\item A \emph{neighborhood of an end} in a non-compact topological space $X$ is a subspace $U \subset X$ which contains a component of $X \setminus K$ for a non-empty compact subspace $K \subset X$.
\item An \emph{end} $e$ of $X$ is an equivalence class of sequences of connected open neighborhoods $X \supset U_1 \supset U_2 \supset \dots$ such that
\[
\bigcap_{i=1}^{\infty} U_i = \emptyset
\]
subject to the equivalence relation
\[
(X \supset U_1 \supset U_2 \supset \dots ) \sim (X \supset V_1 \supset V_2 \supset \dots)
\]
if for each $U_i$ there exists $j$ with $U_i \subseteq V_j$, and for each $V_j$ there exists $i$ with $V_j \subseteq U_i$.
\end{enumerate}
\end{definition}

\begin{remark}[\cite{MR1410261}*{Example~3~(i)}]
\label{rem-exhaust}
Let $X$ be a topological space with a proper map $\varphi \colon X \to [0, +\infty)$ which is onto, and such that the inverse images $U_t = \varphi^{-1}(t, +\infty) \subseteq X$, $t \geq 1$ are connected. Then $X$ has one end.
\end{remark}

\begin{lemma}
\label{lem-exhaust}
Let $X$ be a locally compact Hausdorff space with a countable basis of topology. Then $X$ has one end if and only if there exists an exhaustion by compacts $(K_j)_{j \in \NN}$ of $X$ such that $X \setminus K_j$ is connected for every $j \in \NN$.
\end{lemma}
\begin{proof}
The ``if'' part is straightforward, see also Remark \ref{rem-exhaust}. We only need to provide a proof for the ``only if'' part: Assume to get a contradiction that no exhaustion by compacts $(K_j)_j$ of $X$ exists such that $X \setminus K_j$ is connected for every $j \in \NN$, but that $X$ has only one end. We can always pass to a subsequence, and thus we may assume that $X \setminus K_j$ always has at least two connected components. Moreover, we can assume that none of these components is contained in a compact $K_{\ell(j)}, \ell > j,$ for otherwise we could have included this component already in the compact $K_j$. 
Now we pick a connected component $U_1$ of $X \setminus K_1$. For every $j \geq 2$ we a pick a connected component of $X \setminus K_j$ that is contained in $U_{j-1}$. This is always possible, since we eliminated superficial connected components by our choice of compacts $K_j$. This sequence of neighborhoods $(U_j)_j$ defines an end of $X$.
Since $X \setminus K_1$ has at least two connected components, we choose another sequence of neighborhoods  $(V_j)_j$ such that $U_j$ and $V_j$ are different connected components of $X \setminus K_j$ for every $j \in \NN$, and hence define two different ends.
\end{proof}

\begin{definition}
A \emph{continuum} is a non-empty, compact, connected metric space.
\end{definition}

\begin{lemma}[\cite{MR1192552}*{Theorem 1.8}]
\label{lem-continuum}
The intersection of a decreasing sequence of continua is a continuum.
\end{lemma}

\begin{proposition}
\label{prop-connected}
Let $X$ be a manifold with countable basis of topology and let $\Omega \subset X$ be a relatively compact domain with one end. Then $\Omega$ has connected boundary.
\end{proposition}
\begin{proof}
We apply Lemma \ref{lem-exhaust} to obtain an exhaustion by compacts $K_j$ such that $\Omega \setminus K_j$ is connected for every $j \in \NN$. Then $\Omega \setminus (K_j)^\circ$ are a decreasing sequence of continua. Thus its limit, which is the boundary $\bdry \Omega$, is also a continuum by Lemma \ref{lem-continuum}.
\end{proof}

\begin{remark}
If the boundary $\bdry \Omega$ of the relatively compact domain $\Omega$ admits a collar, then we have a one-by-one correspondence between ends and boundary components.
\end{remark}

\section{Proofs}
\label{sec:psc}

It was already noted by Serre that a Stein manifold of complex dimension at least two has only one end \cite{MR0064155}*{p.~59}, see also \cite{MR3185220}. His short cohomological argument for this fact is given with more details by Gilligan and Huckleberry \cite{MR0616269}*{p.~186}.

Another way of seeing this is to consider a strictly plurisubharmonic Morse exhaustion function which is the approach taken by Forstneri\v{c} \cite{MR1278436}.

\begin{proof}[Proof of Theorem \ref{thm1}]
Let $\dim_\CC X = n$. Since $\Omega$ is relatively compact, by a small perturbation we may assume that the $\cont^2$-smooth exhaustion function $\varphi \colon \Omega \to [0, +\infty)$ is a Morse function and still $q$-complete. The Morse index at a critical point is at most $n + q - 1$ (see the monograph of Forstneri\v{c} \cite{MR3700709}*{Sections 3.10 and 3.11} for more details). Note that $2n - 1 > n + q - 1 \Longleftrightarrow n > q$ is satisfied by assumption. Since we glue only handles of dimension $\leq n + q - 1 $, and $\Omega$ has real dimension $2n$, the complement of any sublevel set of $\varphi$ is connected. By Lemma \ref{lem-exhaust} the domain $\Omega$ has only one end. Proposition \ref{prop-connected} now gives the conclusion.
\end{proof}


\begin{proof}[Proof of Theorem \ref{thm2}]
Let $\dim_\RR X = 2n$. Let $\varphi \colon \Omega \to \RR$ be a strictly $J$-plurisubharmonic exhaustion function. Since the domain $\Omega$ is relatively compact, we may slightly disturb $\varphi$ if necessary to obtain a strictly $J$-plurisubharmonic Morse exhaustion function. By \cite{MR3012475}*{Corollary 3.4} the Morse index of $\varphi$ at a critical point in an almost complex manifold is at most $n$. Now the conclusion is the same as in the proof of Theorem \ref{thm1} above.
\end{proof}

\begin{remark}
The proofs of these two theorems can also be given without using the theory of ends, by instead considering the exhaustion functions and their connected superlevel sets, and applying Lemma \ref{lem-continuum} directly to them. However, conceptually, it seems more natural to use Proposition \ref{prop-connected} which reflects that the ``reason'' comes from the fact that these relatively compact domains have one end.
\end{remark}

\section*{Acknowledgments}

The author would like to thank Matteo Fiacchi, Franc Forstneri\v{c}, Tobias Harz, Gerrit Hermann, Hendrik Hermann, Alexander Izzo, and Jaka Smrekar for interesting and helpful discussions and suggestions. 

\section*{Funding}

The first author was supported by the European Union (ERC Advanced grant HPDR, 101053085 to Franc Forstneri\v{c}) and grant N1-0237 from ARRS, Republic of Slovenia.

\end{document}